\documentclass[11pt, final]{amsart}
\usepackage[margin=1.29in]{geometry}
\usepackage{amsmath, amsxtra, amssymb, mathdots}
\usepackage{appendix}
\usepackage{verbatim}
\usepackage{multirow}

\usepackage[svgnames]{xcolor}
\usepackage{tikz}
\usepgflibrary{decorations.text}
\usetikzlibrary{decorations.text}
\pgfdeclarelayer{edgelayer}
\pgfdeclarelayer{nodelayer}
\pgfsetlayers{edgelayer,nodelayer,main}

\tikzstyle{none}=[inner sep=0pt]
\definecolor{hexcolor0xfefdfd}{rgb}{0.996,0.992,0.992}

\usepackage{graphicx}
\usepackage[cmtip,all]{xy}

\usepackage[notcite, notref]{showkeys}
\usepackage[colorlinks, linkcolor=blue, citecolor=red, urlcolor=black]{hyperref}

\theoremstyle{plain}
\newtheorem*{main-theorem}{Main Theorem}
\newtheorem{theorem}[equation]{Theorem}
\newtheorem{prop}[equation]{Proposition}
\newtheorem{corollary}[equation]{Corollary}

\newtheorem{claim}[equation]{Claim}
\newtheorem*{claim*}{Claim}

\newtheorem{lemma}[equation]{Lemma}

\theoremstyle{definition}
\newtheorem{definition}[equation]{Definition}

\newtheorem{remark}[equation]{Remark}

\numberwithin{equation}{section}

\DeclareMathOperator{\spec}{Spec}

\DeclareMathOperator{\Sym}{Sym}
\DeclareMathOperator{\proj}{Proj}
\DeclareMathOperator{\Stab}{Stab}

\DeclareMathOperator{\SL}{SL}
\DeclareMathOperator{\GL}{GL}
\DeclareMathOperator{\PGL}{PGL}

\DeclareMathOperator{\Grass}{Grass}

\def\cV{\mathcal{V}}
\def\cW{\mathcal{W}}

\def\cU{\mathcal{U}}
\def\cZ{\mathcal{Z}}

\def\irr{\mathrm{irr}}

\def\FF{\mathbb{F}}

\newcommand{\gitq}{/\hspace{-0.25pc}/}
\newcommand\Mg[1]{\overline{\mathcal{M}}_{#1}}
\newcommand\M{\overline{M}}

\def\cL{\mathcal{L}}
\def\cC{\mathcal{C}}
\def\cP{\mathcal{P}}

\def\QQ{\mathbb{Q}}

\def\PP{\mathbb{P}}

\def\CC{\mathbb{C}}
\def\GG{\mathbb{G}}

\def\HH{\mathrm{H}}
\def\bH{\mathbb{H}}
\def\bbH{\overline{\bH}}

\def\cO{\mathcal{O}}
\def\cM{\mathcal{M}}
\def\hra{\hookrightarrow}

\DeclareMathOperator{\Def}{Def}

\DeclareMathOperator{\Aut}{Aut}
\DeclareMathOperator{\Bl}{Bl}

\renewcommand{\bar}{\overline}

\address[Fedorchuk]{Department of Mathematics\\
Boston College\\
140 Commonwealth Ave\\
Chestnut Hill, MA 02467, USA}
\email{maksym.fedorchuk@bc.edu}

\address[Grimes]{Department of Mathematics\\
Boston College\\
140 Commonwealth Ave\\
Chestnut Hill, MA 02467, USA}
\email{matthew.grimes@bc.edu}

\begin{document}

\title{{VGIT presentation of the second flip of $\overline{M}_{2,1}$}}

\author{Maksym Fedorchuk}
\author{Matthew Grimes}

\thanks{The first author was partially supported by the NSA Young Investigator grant H98230-16-1-0061 
and Alfred P. Sloan Research Fellowship.}

\begin{abstract}
We perform a variation of geometric invariant theory stability analysis 
for 2nd Hilbert points of bi-log-canonically embedded pointed curves of genus two.
As a result, we give a GIT construction of the log canonical models
$\M_{2,1}(\alpha)$ for $\alpha=2/3\pm \epsilon$, and obtain a VGIT presentation of the second flip 
in the Hassett-Keel program 
for the moduli space of pointed genus two curves.
\end{abstract}

\maketitle

\section{Introduction}
The goal of the Hassett-Keel program is to inductively construct 
stacks $\Mg{g,n}(\alpha)$ of singular curves whose moduli spaces are the following 
log canonical models of the stack $\Mg{g,n}$ of Deligne-Mumford stable curves:
\[
\overline{M}_{g,n}(\alpha)=\proj \bigoplus_{m=0}^{\infty} 
\HH^0\bigl(\overline{\mathcal M}_{g,n},m(K_{\overline{\mathcal M}_{g,n}}+\alpha\delta+(1-\alpha)\psi)\bigr).
\]

Several stages of this program, over the field $\CC$ of complex numbers,
have been worked out to date in \cite{hassett-hyeon2009,hassett2013log,afsflip-1} for arbitrary $g$, 
and can be summarized by the following diagram (see \cite[Main Theorem, p.3]{afsflip-1} for precise
statements and notation):

\vspace{1pc}

\makebox[\textwidth][c]{
\xymatrix{
\Mg{g,n} \ar@{^(->}[r] \ar[d]	& \Mg{g,n}(\frac{9}{11}) \ar[dd] & \Mg{g,n}(\frac{9}{11}-\epsilon) 
\ar@{^(->}[r] 
\ar[d] \ar@{_(->}[l] & \Mg{g,n}(\frac{7}{10}) \ar[dd] & \Mg{g,n}(\frac{7}{10}-\epsilon) 
\ar@{_(->}[l]
\ar@{^(->}[r] \ar[d] & \Mg{g,n}(\frac{2}{3}) \ar[dd]& \Mg{g,n}(\frac{2}{3}-\epsilon) 
\ar@{_(->}[l]  \ar[d]  \\
\M_{g,n}\ar[rd]	& &\M_{g,n}(\frac{9}{11}-\epsilon) \ar[rd] \ar[ld]	
& & \M_{g,n}(\frac{7}{10}-\epsilon)\ar[rd] \ar[ld] & & \M_{g,n}(\frac{2}{3}-\epsilon)\ar[ld] \\
& \M_{g,n}(\frac{9}{11})& & \M_{g,n}(\frac{7}{10})& & \M_{g,n}(\frac{2}{3}) & 
}
}

Geometric Invariant Theory (GIT) constructions were given for $\Mg{g}(\alpha)$ for $\alpha > 2/3$ by Hassett and Hyeon \cite{hassett-hyeon2009,hassett2013log} using Chow and asymptotic Hilbert stability analysis of pluricanonically embedded curves.
The construction by \cite{afsflip-1} of the next step (the wall-crossing from $\alpha=7/10-\epsilon$ to $\alpha=2/3-\epsilon$) 
in the Hassett-Keel program
required the machinery of stacks and good moduli spaces, 
and until now no GIT construction of this step was known for any 
$(g,n)$.  As one of our main results in this paper, we obtain such a GIT construction for $(g,n)=(2,1)$.
Namely, restricting in the above diagram to $\alpha<7/10$ and taking $g=2$ and $n=1$, we have
the following picture:
\vspace{1pc}
\begin{equation}
	\label{diag:second-flip}
	\begin{gathered}
\xymatrix{
\Mg{2,1}(\frac{7}{10}-\epsilon) 
\ar@{^(->}[r] \ar[d]	& \Mg{2,1}(\frac{2}{3}) \ar[dd] & \Mg{2,1}(\frac{2}{3}-\epsilon) 
\ar@{_(->}[l] \ar[d] \\
\M_{2,1}(\frac{7}{10}-\epsilon)\ar[rd] & & \M_{2,1}(\frac{2}{3}-\epsilon)\ar[ld]_{\simeq} \\
 & \M_{2,1}(\frac{2}{3}) & 
 }
\end{gathered}
\end{equation}
where $\M_{2,1}(\frac{2}{3}-\epsilon) \to \M_{2,1}(\frac{2}{3})$ is an isomorphism of projective good moduli 
spaces, and $\M_{2,1}(\frac{7}{10}-\epsilon) \to \M_{2,1}(\frac{2}{3})$ is a birational 
contraction of the Weierstrass divisor (see \cite[Remark 1.1]{afsflip-1}).
Our main result can now be stated (somewhat vaguely) as follows:
\begin{main-theorem}
\label{T:intro} 
	There is a projective master space $\overline{\mathbb{H}}$ over $\spec \CC$ 
	with a $\PGL(5)$-action
and a family of ample linearizations $\mathcal{L}_{\beta}$ 
such that Diagram \eqref{diag:second-flip} is one of the two VGIT wall-crossings relating 
GIT quotients $\overline{\mathbb{H}}\gitq_{\cL_{\beta}} \PGL(5)$. 
\end{main-theorem}
We give a fuller and more technical statement of this result in Theorem \ref{T:maintheorem} after all necessary notation is introduced. We note that the 
second VGIT wall-crossing from the above theorem gives rise to a new moduli stack
of genus $2$ pointed curves with at worst $A_5$-singularities, whose moduli space is a single point.

We have several reasons for focusing on $\M_{2,1}$. The first is largely historical: Hassett's paper \cite{Hassett-genus2} that laid foundations for the Hassett-Keel program was devoted to the log minimal model program for 
$\M_{2}$. The program is also complete for $\M_{3}$ \cite{hyeon-lee_genus3}, and so we naturally wondered
what happens for $\M_{2,1}$, a space whose complexity sits between that of $\M_{2}$ and $\M_{3}$. 
The second reason is that, as we have discovered in the course of this project, 
the GIT approach to the Hassett-Keel program for $\M_{2,1}$ requires unorthodox constructions,
which are unexpected from the point of view of classical GIT constructions of moduli spaces of curves.

To illustrate one such intricacy, we recall that in the case of pointed curves, we have a natural variation of the GIT
problem due to linearizations coming from the parameter space of embedded curves
and from the space of points.  Since the moduli space $\M_{2,1}(7/10-\epsilon)$ was constructed
by Hyeon and Lee \cite{hyeon-lee_genus4} as a GIT quotient where the parameter space of embedded curves 
was a Hilbert scheme of degree $6$
genus $2$ curves, one would expect that by varying the linearization in their GIT setup, 
one could obtain the next step 
in the Hassett-Keel program, namely $\M_{2,1}(2/3)$. This turned out not to be the case, 
as explained in \S\ref{S:prior}. Instead, our construction requires a
delicate analysis of finite Hilbert stability (cf. \cite{AFS-stability}) of degree $6$ genus $2$ curves, 
and, as we will explain in the sequel, 
is impossible to replicate using the more classical Chow 
or asymptotic Hilbert stability.

\subsection*{Roadmap of the paper} In Section \ref{S:geometry-genus-2}, we discuss various
stability conditions 
for pointed genus $2$ curves with at worst type $A$ (i.e., $y^2=x^n$) singularities, introduce
bi-log-canonical curves, and 
their realizations as quadric sections of rational normal surface scrolls in $\PP^4$. 
In Section \ref{S:VGIT}, we set up our VGIT
problem, state our main result (Theorem \ref{T:maintheorem}), 
and discuss its relation to prior works.  The proof of Theorem \ref{T:maintheorem} is given 
in Section \ref{S:proof}.

We work over the field $\CC$ of complex numbers throughout.

\section{Geometry of genus $2$ curves with $A$-singularities}
\label{S:geometry-genus-2}
In what follows, we collect together facts about the geometry of 
pointed genus $2$ curves that are needed for our VGIT stability analysis;
while these results are all well-known in the case of smooth curves, 
their extension to the case of mildly singular curves, e.g., to curves with higher $A$-singularities,
including those introduced in \cite{afsflip-1}, requires some care. 

We will say that $(C,p)$ is a \emph{pointed genus $2$ curve} if 
$C$ is a projective Gorenstein curve of arithmetic genus $2$,
and $p\in C$ is its smooth point. We will say 
that $p$ is a \emph{Weierstrass point} of $C$ if $\omega_C\simeq \cO_C(2p)$. 

We now introduce five notions of stability for pointed genus $2$ 
curves.
The first four of these appeared 
in \cite{afsflip-1}, 
whose definitions of elliptic tails and bridges we keep (see \cite[Definition 2.1]{afsflip-1}),
while the second was introduced in \cite{hyeon-lee_genus4}.
\begin{definition}
\label{D:stability}
Suppose that $(C,p)$ is a pointed genus $2$ curve with 
$\omega_{C}(p)$ ample. 
Then we will say that $(C,p)$ is:
	\begin{itemize}
	\item \emph{$A_2$-stable}: 
		if $C$ has only $A_1$, $A_2$-singularities,
		and $C$ has no nodally attached elliptic tails;
		\item \emph{$A_3$-stable}: 
		if $C$ has only $A_1$, $A_2$, $A_3$-singularities, and $C$ has no nodally or tacnodally attached elliptic tails, 
		and no nodally attached elliptic bridges;
		\item \emph{$A_4$-stable}: if $C$ has only $A_1$, $A_2$, $A_3$, $A_4$-singularities,
		and $C$ has no 
		nodally or tacnodally attached elliptic tails or nodally attached elliptic bridges; 
		\item \emph{$A_4^{non-W}$-stable}: if $C$ is $A_4$-stable, and $p$ is not a Weierstrass point.
		\item \emph{$A_5$-stable}: if  $C$  has only $A_1$, $A_2$, $A_3$, $A_4$, $A_5$-singularities,
$C$ has no $A_1$ or $A_3$-attached elliptic tails or nodally attached elliptic bridges,
and $p$ is not a Weierstrass point of $C$.
	\end{itemize}
	Finally, we will say that $(C,p)$ is a \emph{Weierstrass curve} if $(C,p)$ is $A_4$-stable and $p$ is a Weierstrass point. In other words, a Weierstrass curve is an $A_4$-stable curve that is not $A_4^{non-W}$-stable.
\end{definition}

These five notions of stability are modifications of the standard 
Deligne-Mumford stability \cite{DM}, and arise naturally in the Hassett-Keel 
program for $\M_{2,1}$.  The moduli space of $A_3$-stable pointed genus 
$2$ curves was first constructed in \cite{hyeon-lee_genus4} in the course of 
their analysis of the Hassett-Keel program for $\M_4$. 
The (generalizations of) moduli stacks of 
$A_2$, $A_3$, $A_4$, and $A_4^{non-W}$-stable curves, and their good moduli spaces, 
were constructed for an arbitrary genus and an arbitrary number of marked points in 
\cite{afsflip-1} using stack-theoretic techniques. We note that in the terminology of \cite{afsflip-1}, $A_3$-stable curves are 
$\left(\frac{2}{3}+\epsilon\right)$-stable, $A_4$-stable curves are $\frac{2}{3}$-stable, and, as we will see shortly in Corollary \ref{C:omega}, 
$A_4^{non-W}$-stable curves are $\left(\frac{2}{3}-\epsilon\right)$-stable.  

It was established in 
\cite[Theorem 2.7]{afsflip-1}, that the stacks $\Mg{2,1}(\alpha)$ of 
$\alpha$-stable curves are algebraic, and there are open immersions of stacks
\begin{equation}\label{D:stack-flip}
\overline{\cM}_{2,1}\left(\frac{2}{3}+\epsilon\right)\hookrightarrow\overline{\cM}_{2,1}\left(\frac{2}{3}\right)\hookleftarrow
	\overline{\cM}_{2,1}\left(\frac{2}{3}-\epsilon\right).
\end{equation}
Our goal is to give a VGIT presentation for Diagram \eqref{D:stack-flip}. 
Our approach to a GIT construction of the moduli stacks $\Mg{2,1}(\alpha)$ 
is to consider genus $2$ curves equipped with a bi-log-canonical embedding,
which we discuss below in \S\ref{S:bicanonical}.

\begin{remark}
A moment of reflection will convince the reader that the only $A_5$-stable but not $A_4^{non-W}$-stable
curve is isomorphic to a union of two smooth rational curves along an $A_5$-singularity, with a marked point 
on one of the components, and that all such curves are isomorphic. We will prove later on that this isomorphism class
gives a unique closed point in the stack of $A_5$-stable curves (see Lemma \ref{L:A5-curve}).
\end{remark}

It is clear that both $A_3$-stability and 
$A_4^{non-W}$-stability imply $A_4$-stability, and that $A_4^{non-W}$-stability implies $A_5$-stability.
Before we proceed, we will need the following technical result  about
$A_4$ and $A_5$-stable curves:
\begin{prop}\label{P:3-connected} Every $A_4$-stable and $A_5$-stable pointed genus 
$2$ curve is \emph{$3$-connected} in the sense of 
\cite[Definition 3.1]{embeddings-curves}.
\end{prop}
\begin{proof} 
Suppose $B\subset C$ is a strict subcurve such that $\deg \omega_{C}\vert_B-\deg \omega_B=1$.
Then $B$ meets $\overline{C\setminus B}$ in a node and necessarily $p_a(B)=1$. 
If $B$ is unpointed, then $B$ is an elliptic tail; otherwise, $B$ is an elliptic bridge. 

Suppose $B\subset C$ is a strict subcurve such that $\deg \omega_{C}\vert_B-\deg \omega_B=2$.
If $p_a(B)=0$, then necessarily $B$ is pointed.  It follows that $\overline{C\setminus B}$ has either 
an elliptic bridge, or an $A_3$-attached elliptic tail. If $p_a(B)=1$, then $p_a(\overline{C\setminus B})=0$ 
and so $B$ is unpointed. It follows that $B$ is either an elliptic bridge, or an $A_3$-attached elliptic tail.
\end{proof}
\begin{remark} The above result fails in higher genus, and in fact already
in genus $3$ as illustrated by a union of a smooth genus $2$ curve and 
a pointed rational curve meeting along an $A_3$-singularity.
\end{remark}
\begin{corollary}\label{C:omega} Suppose $(C,p)$ is an $A_4$-stable or $A_5$-stable curve. 
Then $C$ is \emph{honestly hyperelliptic},
that is $\vert \omega_C\vert$ is a base-point-free linear system defining a finite 
flat degree $2$ map 
$\kappa_C\colon C \to \PP^1$. In particular, a point $p\in C$ is a Weierstrass point if and only if $p$
is a ramification point of the canonical morphism $\kappa_C$, if and only if
$(C,p)$ is a Weierstrass curve.
\end{corollary}
\begin{proof} This follows from Proposition \ref{P:3-connected} and \cite[Theorem 3.6]{embeddings-curves}.
\end{proof}

\subsection{Bi-log-canonical embeddings of $A_4$ and $A_5$-stable curves}
\label{S:bicanonical}

Given an $A_4$-stable or $A_5$-stable pointed genus $2$ curve $(C,p)$, 
the line bundle $\omega^2_C(2p)$ 
will be called a \emph{bi-log-canonical line bundle}.
In analogy with the standard case of smooth curves, we have the following
simple result:
\begin{lemma}\label{L:very-ample} The line bundle $\omega^2_C(2p)$ is very ample for every $A_4$-stable or $A_5$-stable
pointed curve $(C,p)$.
\end{lemma}
\begin{proof} By \cite[Theorem 1.1]{embeddings-curves}, it suffices to 
verify that $\deg \omega^2_C(2p)\vert_{B} \geq 2p_a(B)+1$ for every strict subcurve $B\subset C$. This is straightforward using the definition of 
$A_4$ or $A_5$-stability.
\end{proof}

A choice of an isomorphism $V \simeq \HH^0\bigl(C, \omega^2_C(2p)\bigr)$
then gives an embedding $\phi\colon C\hra \PP V^{\vee}$ whose image 
is a degree $6$ curve in $\PP V^{\vee} \simeq \PP^4$, which we call
a \emph{bi-log-canonical curve}. 
We denote
by $I_C$ the homogeneous ideal of $\phi(C)$.
It is standard to verify that $\phi(C)$ is projectively normal, and so we can define
\emph{the $m^{th}$ Hilbert point of $C$} to be the short exact sequence
\[
0 \to (I_C)_m \to \Sym^m V \to \HH^0\bigl(C, \omega^{2m}_C(2mp)\bigr) \to 0
\]
considered as a point in the Grassmannian
\[
\Grass\bigl((I_C)_m, \Sym^m V\bigr).
\]
Our focus will be entirely on the 2nd Hilbert points of bi-log-canonically 
embedded curves, and so we note that for $m=2$, we have $\dim \Sym^2 V = 15$, 
while $h^0(C,\omega^{4}_C(4p)) = 11$,  
and $\dim (I_C)_2=4$.

\subsection{Surfaces containing bi-log-canonical curves}
 \label{S:surfaces-containing-curves}
 
 Our analysis of GIT semistability of 2nd Hilbert points of bi-log-canonical curves is rooted in the analysis of minimal surfaces on which 
 these curves and their degenerations lie, and so we pause to describe them.
Whether the marked point $p$ is a Weierstrass point on $C$ determines on
which minimal surface the bi-log-canonical curve lies. 
\subsubsection{Non-Weierstrass curves} 
First, we see that  
$A_5$-stable (and hence $A_4^{non-W}$-stable)
curves have planar models given 
by cuspidal plane quartics.  This observation and the existence of a GIT moduli space for plane quartics will be essential in our analysis of GIT stability for these curves.
\begin{lemma}\label{L:cuspidal-quartic}
Suppose that $(C,p)$ is an $A_5$-stable curve.
Then the line bundle 
$\omega_C(2p)$ is base-point-free and defines a morphism $\psi\colon 
C \to \PP^2$ such that $\psi(C)$ is a plane quartic. Moreover, $\psi$ is an isomorphism onto its image away from 
$p$ and $\psi(p)$ is a cusp of $\psi(C)$. 
\end{lemma} 

\begin{proof}
Let $\bar C$ be the unique arithmetic genus $3$ curve obtained from $C$ by imposing 
a cusp at $p$. Since $p$ is not a Weierstrass point of $C$, we have that 
$\bar C$ is not strictly hyperelliptic. It follows by \cite[Theorem 3.6]{embeddings-curves}
that $\omega_{\bar C}$ is very ample and embeds $\bar C$ as a quartic curve in 
$\PP^2$. The claim follows by noting that the partial normalization morphism 
$\nu \colon C\to \bar C$ satisfies $\nu^{*}\omega_{\bar C}=\omega_C(2p)$ and 
that $h^0\bigl(C, \omega_C(2p)\bigr)=3$. 
\end{proof}
\begin{definition}\label{D:cuspidal-quartic} 
We call the image $\psi(C)$ \emph{the cuspidal quartic model of $(C,p)$}.
\end{definition}

Let $\psi(C) \hookrightarrow \PP^2$ be the cuspidal 
quartic model of $(C,p)$. 
Blowing-up the cusp $q:=\psi(p)$ in $\PP^2$ 
realizes $C$ as a smooth curve in the 
class $|4H-2E|$ on the Hirzebruch surface $\Bl_{q} \PP^2\simeq \FF_1$,
 tangent to the exceptional divisor $E$ at the point $p\in C$. 
 The embedding of $\FF_1$ by the complete linear system $\vert 2H-E\vert$ 
 realizes the surface as a \emph{cubic surface scroll} $S_1 \hookrightarrow \PP^4$ 
 (the reader can find an overview of the properties of 
 cubic surface scrolls in \cite[\S6.1]{harris-roth-starr}, which we use in what follows without further comments).
 The directrix of $S_1$ is the line given by the image of $E$, and for
this embedding of $S_1$ into $\PP^4$, the curve $C$ becomes a quadric section of 
$S_1$ tangent to the directrix at $p\in C$. 
 
 Very concretely, we can realize 
 $S_1$ as the image of the rational map 
 $\PP^2 \dashrightarrow \PP^4$ given by 
 \begin{equation}\label{scroll-embedding}
 [x:y:z] \mapsto [xz: yz: x^2: xy: y^2], 
 \end{equation}
 which leads to a determinantal presentation of the quadric generators of the ideal of $S_1$.
 Namely, $(I_{S_1})_{2}$ is given by the $2\times 2$ minors of
 \[
 \left(\begin{matrix} z_0 & z_2 & z_3 \\ z_1 & z_3 & z_4\end{matrix}\right).
 \]
We denote the resulting net of quadrics by $N_1$, so that  
 \begin{equation}\label{E:N1}
 N_1:=(z_0z_3 -z_1z_2, z_0z_4-z_1z_3, z_2z_4-z_3^2)\in \Grass(3, \Sym^ 2V).
 \end{equation}
The equation of the directrix of $S_1$ is $z_2=z_3=z_4=0$ in these coordinates.
 
 \subsubsection{Weierstrass curves}
 \label{S:weierstrass}
 Suppose $(C,p)$ is an $A_4$-stable curve and $p$ is its Weierstrass point.
 Since $C$ is honestly hyperelliptic by Corollary \ref{C:omega}, it admits a finite flat degree $2$ map to 
 $\PP^1$, 
 and so we can represent the affine curve $C\setminus\{p\}$ by the equation
 \begin{equation}\label{affine-Weierstrass}
 y^2=x^5+c_3x^3+c_2x^2+c_1x+c_0, \quad \text{$c_i\in \CC$},
 \end{equation}
and take $p$ to be the point at infinity.  It is easy to check that 
 $dx/y$ is a Rosenlicht differential on $C$ with a double zero at infinity, and so is a global regular section of $\omega_C(-2p)$. 
 Note that $x$ has a double pole at infinity and is regular everywhere else, 
 and $y$ has
 a pole of order $5$ at infinity and is regular everywhere else.
 It follows that 
 \begin{equation}\label{E:bi-log-basis}
 \HH^0\bigl(C,\omega_C^2(2p)\bigr)=
  \left\langle x^3\left(\frac{dx}{y}\right)^2, \ y\left(\frac{dx}{y}\right)^2, \  x^2\left(\frac{dx}{y}\right)^2, \ x\left(\frac{dx}{y}\right)^2, \ \left(\frac{dx}{y}\right)^2 \right\rangle.
 \end{equation}
 Denoting the listed generators of $\HH^0\bigl(C,\omega_C^2(2p)\bigr)$
 by $z_0,\dots, z_4$, we see that 
 the resulting homogeneous ideal of the bi-log-canonical curve $\phi(C)$ in $\PP^4$ is 
 \begin{equation}\label{E:Weierstrass}
 (z_0z_3 -z_2^2, z_0z_4-z_2z_3, z_2z_4-z_3^2, z_1^2-z_0z_2-c_3z_2z_3-c_2z_2z_4-c_1z_3z_4-c_0z_4^2),
 \end{equation} 
and the marked point is $p=[1:0:0:0:0]$.
The first three quadrics in \eqref{E:Weierstrass} cut out a singular minimal surface in $\PP^4$ given by a  
\emph{cone over a rational normal cubic curve}. We call this surface $S_2$.
The quadric generators of the homogeneous ideal of $S_2$ are given by the net 
 \begin{equation}\label{E:N2}
 N_2:=(z_0z_3 -z_2^2, z_0z_4-z_2z_3, z_2z_4-z_3^2) \in \Grass(3, \Sym^ 2V),
 \end{equation}
 which is of course given by the $2\times 2$ minors of 
  \[
 \left(\begin{matrix} z_0 & z_2 & z_3 \\ z_2 & z_3 & z_4\end{matrix}\right).
 \]
 
 \subsection{The ramphoid cuspidal atom}
  \label{S:atom}

Next, we introduce a special Weierstrass curve, called \emph{the ramphoid cuspidal atom}, and denoted by $C_R$. 
This curve is defined by the following equation
 \begin{equation}\label{E:RamphoidCusp}
  I_{C_R}=(z_0z_3 -z_2^2, z_0z_4-z_2z_3, z_2z_4-z_3^2, z_1^2-z_0z_2).
 \end{equation}
As above, we mark the point $\infty:=[1:0:0:0:0]\in C_R$.  Note that $C_R$ is simply a compactification of
the affine curve $y^2=x^5$ given by Equation \eqref{affine-Weierstrass} with $c_3=c_2=c_1=c_0=0$, and 
marked by the point at infinity. It is evident that the connected component of the identity 
in the automorphism group of 
$(C_R, \infty)$ is $\Aut(C_{R},\infty)^{\circ}\simeq \GG_m$,
where $\GG_m$ acts by 
\begin{align*}
&t\cdot (x,y) =(t^2 x, t^5 y). 
\end{align*}
In particular, the 1-PS of $\GL(5)$ (respectively, the induced 1-PS of $\SL(5)$) 
acting on the generators $(z_0,\dots, z_4)$ of $\HH^0\bigl(C_R,\omega_{C_R}^2(2\infty)\bigr)$
as given by \eqref{E:bi-log-basis} with weights 
$(0,-1,-2,-4,-6)$ (resp., $(13, 8, 3, -7, -17)$) 
is the connected component of the identity in the stabilizer of $(C_{R}, \infty)$ in $\GL(5)$
(resp., $\SL(5)$).

We will make use of the first-order deformation 
theory of $(C_{R}, \infty)$, and so we pause to describe it here.  Let $\Def^1(C_{R},\infty)$ and $\Def^1(A_4)$ be the
first-order deformation spaces of the pointed curve $(C_{R}, \infty)$ and the $A_4$-singularity, respectively.  By 
standard deformation theory, we
have $\dim_{\CC} \Def^1(C_{R},\infty)=5$
and $\dim_{\CC} \Def^1(A_4)=4$. 
Furthermore, we have a short exact sequence
\begin{equation}\label{E:def-A4}
0 \to T_{\infty}\to \Def^1(C_{R},\infty) \to \Def^1(A_4) \to 0,
\end{equation}
where $T_{\infty}$ is the tangent space to $C_{R}$ at $\infty$ naturally identified with the first-order deformations of the 
point $\infty$ in $C_{R}$.
(A priori, $T_{\infty}$ is only a subspace of the kernel of the surjection $\Def^1(C_{R},\infty) \to \Def^1(A_4)$, but in this case 
exactness follows by dimension considerations.)
Let $W$ be the tangent space to the locus of Weierstrass curves (that is, those given by Equation \eqref{affine-Weierstrass}) 
in the miniversal deformation
space of $(C_{R},\infty)$. Since Equation \eqref{affine-Weierstrass} also describes the miniversal deformation 
of the $A_4$-singularity, we conclude that the subspace
$W \hookrightarrow \Def^1(C_{R},\infty)$ maps isomorphically onto $\Def^1(A_4)$ in \eqref{E:def-A4}.
We also note that  \eqref{E:def-A4} is a  short exact sequence of $\GG_m$-representations, where
 $\GG_m\simeq \Aut(C_{R},\infty)^{\circ}$ acts on $\Def^1(A_4)$ with weights $(-4, -6, -8, -10)$, and on $T_{\infty}$ with weight $1$.

\section{VGIT problem and main theorem}
\label{S:VGIT}

As discussed in \S\ref{S:bicanonical}, every $A_4$ or $A_5$-stable pointed genus $2$ curve $(C,p)$ is
equipped with a bi-log-canonical embedding $C \hookrightarrow \PP V^{\vee}$, where 
$V \simeq \HH^0\bigl(C, \omega_C^2(2p)\bigr)\simeq \CC^5$. By 
\S\ref{S:surfaces-containing-curves}, 
to each such bi-log-canonical curve $C\hra \PP^4$ we can associate a unique
(either smooth or singular) cubic surface scroll $S$ containing it.  We now introduce 
an element $h(C,p)$ of $\Grass(4, \Sym^2 V) \times \PP V^{\vee}$, which we call \emph{the H-point of $(C,p)$}, defined by
\begin{equation}\label{H-point}
h(C,p):=\bigl((I_C)_2, p\bigl)\in \Grass(4, \Sym^2 V) \times \PP V^{\vee},
\end{equation}
where $[(I_C)_2 \subset \Sym^2 V] \in \Grass(4, \Sym^2 V)$ is the 2nd Hilbert point of $C$.
Note that we always have 
\begin{equation}\label{net-quadrics}
(I_S)_2 \subset (I_C)_2, \ \text{where either $(I_S)_2=N_1$ or $(I_S)_2=N_2$, up to $\SL(V)$-action.}
\end{equation}

Denote by $\bH^{sm}$ 
the constructible locus inside 
$\Grass(4, \Sym^2 V) \times \PP V^{\vee}$ 
consisting of all H-points $h(C,p)$, as $(C,p)$ varies over all  
{smooth} 
bi-log-canonically embedded pointed genus $2$ curves, 
and by $\bbH$ the Zariski closure of $\bH^{sm}$. 
The standard action of $\SL(V)\simeq \SL(5)$ on 
$\Grass(4, \Sym^2 V) \times \PP V^{\vee}$ restricts to an action 
on $\bbH$, where the orbits in $\bH^{sm}$ 
are in bijection with abstract 
isomorphism classes of smooth 
pointed genus $2$ curves.

We will construct a
compactification of $M_{2,1}$ by taking a GIT quotient of $\bbH$. Moreover,
since $h(C,p)$ is well-defined for every bi-log-canonically embedded $A_4$-stable curve $C$, 
we will obtain the moduli space of $A_4$-stable curves as a GIT quotient of $\bbH$ after choosing an
appropriate linearization of the $\SL(5)$-action on $\bbH$. Note that
because the space of ample linearizations of the $\SL(5)$-action on $\bbH$
contains a two-dimension subspace spanned by the pullbacks of 
$\cO_{\Grass(4,\Sym^2 V)}(1)$ and $\cO_{\PP V^\vee}(1)$, 
we have a natural VGIT problem (see \cite{dolgachev-hu} and \cite{thaddeus-VGIT} 
for an introduction to VGIT), which we proceed to describe. 

For an ample $\QQ$-line-bundle 
\begin{equation}\label{beta-linearization}
\cL_{\beta}:=\cO_{\Grass(4,\Sym^2 V)}(1) \boxtimes\cO_{\PP^4}(\beta), \quad \text{where $\beta\in (0,\infty)\cap \QQ$,} 
\end{equation}
we let $\bbH(\beta)$ be the semistable locus inside 
$\bbH$ with respect to the linearization given by $\cL_{\beta}$. 
The points of $\bbH(\beta)$ will be called \emph{$\beta$-stable pairs}. 
We then have a sequence of GIT quotients
\[
\overline{H}_{\beta}:=\bbH(\beta) \gitq_{\cL_\beta} \SL(5).
\] 

Our first goal is to show that 
for a range of $\beta$-values, 
the semistable locus $\bbH(\beta)$ parameterizes only curves. Namely,
we will show that for $\beta \in (1/7, 1/2]$, every $\beta$-stable pair $(I, p)\in \Grass(4, \Sym^2 V)\times \PP^4$
is such that $(I, p)=h(C,p)$ for some
bi-log-canonically embedded curve $(C,p)$.
Incidentally, we will see that for $\beta \in (1/7, 1/2)$, the semistable locus 
$\bbH(\beta)$ consists exactly of H-points of $A_4$-stable curves, 
thus establishing a posteriori that the locus of H-points of $A_4$-stable curves 
is locally closed in $\Grass(4, \Sym^2 V)\times \PP^4$. 

Our second goal is to explicitly describe curves parameterized by 
$\bbH(\beta)$ for all $\beta \in (1/7, 1/2]$. Our third
and final goal is to show that the resulting GIT quotients are identified 
with the moduli spaces of $A_3$, $A_4$, and $A_4^{non-W}$-stable curves, respectively, thus 
giving a VGIT presentation for the 2nd flip of $\M_{2,1}$.

We denote by $\bbH(a,b)$ the locus of $\beta$-stable pairs if the stability does not change as 
$\beta$ varies in the interval $(a,b) \subset (0,\infty)\cap \QQ$. With this, we can state our main theorem as follows:

\vspace{1pc}

\begin{theorem} 
	\label{T:maintheorem}
For $\beta\in (1/7, 1/2]$, the $\beta$-stable pairs have the following description:
	\begin{enumerate}
\item[(A)] We have that 
\[
\bbH(1/7, 4/13) =\bigl\{h(C,p)\mid \ \text{$(C,p)$ is a bi-log-canonically embedded $A_3$-stable curve}\bigr\},
\]
and, consequently,
\[
\bigl[\bbH(1/7, 4/13) / \PGL(5)\bigr] \simeq \Mg{2,1}(2/3+\epsilon).\]
\item[(B)] 
We have that  
\[
\bbH(4/13) =\bigl\{h(C,p)\mid \ \text{$(C,p)$ is a bi-log-canonically embedded $A_4$-stable curve}\bigr\},
\]
and, consequently,
\[
\bigl[\bbH(4/13) / \PGL(5)\bigr] \simeq \Mg{2,1}(2/3).\]
\item[(C)]
We have that 
\[
\bbH(4/13,1/2) =\left\{h(C,p)\mid \ \text{$(C,p)$ is a bi-log-canonically embedded $A_4^{non-W}$-stable curve}\right\},
\]
and, consequently,
\[
\left[\bbH(4/13,1/2) / \PGL(5)\right] \simeq \Mg{2,1}(2/3-\epsilon).
\]
\item[(D)] 
We have that 
\[
\bbH(1/2) =\bigl\{h(C,p)\mid \ \text{$(C,p)$ is a bi-log-canonically embedded $A_5$-stable curve}\bigr\},
\]
and, consequently,
\[
\left[\bbH(1/2) / \PGL(5)\right] 
\]
is the moduli stack of $A_5$-stable curves.

\item[(E)] The moduli spaces of the GIT quotient stacks described in 
Parts (A), (B), (C), (D) are isomorphic to the log canonical
models of $\M_{2,1}$ as follows:
\begin{align}
\bbH(1/7,4/13) \gitq \SL(5) & \simeq \M_{2,1}(2/3+\epsilon), \\
\bbH(4/13) \gitq \SL(5)  &\simeq \M_{2,1}(2/3), \label{lc-2/3}\\
\bbH(4/13,1/2) \gitq \SL(5)  &\simeq \M_{2,1}(2/3-\epsilon), \\
\bbH(1/2) \gitq \SL(5)  &\simeq \M_{2,1}(19/29) \simeq \{\text{point}\},
\end{align}
and fit into a commutative diagram
\begin{equation}\label{E:VGIT}
\begin{aligned}
\xymatrix{
\bbH(4/13-\epsilon) \ar@{^(->}[r] \ar[d]	
& \bbH(4/13) \ar[dd] & \bbH(4/13+\epsilon) \  \ar[d]  \ar@{_(->}[l]   \ar@{^(->}[r] & \bbH(4/13) \ar[dd] \\
\M_{2,1}(2/3+\epsilon)\ar[rd]
& 			
&  \M_{2,1}(2/3-\epsilon)\ar[dr] \ar[dl]_{\simeq}	\\
&	   \M_{2,1}(2/3)     &  & \M_{2,1}(19/29)=\{\text{point}\} \\
}
\end{aligned}
\end{equation}
The birational morphism $\M_{2,1}(2/3+\epsilon) \to \M_{2,1}(2/3)$ contracts the Weierstrass divisor to a point,
while $\M_{2,1}(2/3-\epsilon) \to \M_{2,1}(2/3)$ is an isomorphism.
\end{enumerate}
\end{theorem}

\begin{remark} We emphasize that the identification of the GIT quotient stacks
and the stacks of $\alpha$-stable curves in Parts (A--D) of Theorem \ref{T:maintheorem} are formal
consequences of the GIT stability analysis.  Namely, suppose 
we have established for some two values of $\alpha$ and $\beta$ that
\begin{equation}\label{E:id}
\bbH(\beta)=\left\{h(C,p)\mid \ \text{$(C,p)$ is a bi-log-canonically embedded $\alpha$-stable 
curve}\right\}.
\end{equation}
We claim this implies that
\[
\bigl[\bbH(\beta) / \PGL(5)\bigr] \simeq \Mg{2,1}(\alpha).
\]
Indeed, by \eqref{E:id}, we have a flat family of $\alpha$-stable curves over $\bbH(\beta)$.
This family induces a $\PGL(5)$-equivariant morphism $\bbH(\beta)\to \Mg{2,1}(\alpha)$
that descends to give a morphism $f\colon [\bbH(\beta)/\PGL(5)] \to \Mg{2,1}(\alpha)$.
Conversely, let $\pi\colon \cC\to \Mg{2,1}(\alpha)$ be the universal family of $\alpha$-stable curves,
with the universal section $\sigma$.
Let $\cV:=\pi_*\left(\omega^2_{\pi}(2\sigma)\right)$, which is a vector bundle of rank $5$. 
By Lemma \ref{L:very-ample}, we have a morphism $\cC \to \PP(\cV)$ over $\Mg{2,1}(\alpha)$ 
given by a relatively 
very ample line bundle $\omega_{\pi}^2(2\sigma)$. 
Let $\cP\to \Mg{2,1}(\alpha)$ be the $\PGL(5)$-torsor associated to $\PP(\cV)$. 
Then the pullback of $\PP(\cV)$ to $\cP$ is a trivial $\PP^4$-bundle and so we obtain 
a morphism $\cP\to \bbH(\beta)$ given by taking the $H$-points of the fibers in $\PP^4$. 
This defines a morphism $g\colon \Mg{2,1}(\alpha)
\to [\bbH(\beta)/\PGL(5)]$. That $g\circ f$ and $f\circ g$ are identities follows from the 
fact that for $\alpha$-stable curves $(C,p)$ and $(C',p')$, we have $(C,p) \simeq (C',p')$ if and only if $\PGL(5)\cdot h(C,p)=\PGL(5)\cdot h(C',p')$.
\end{remark}

We summarize the relationship between the 
notions of stability from Definition \ref{D:stability} (and \cite{afsflip-1})
and the GIT stability for H-points of bi-log-canonical curves in Table \ref{table-replacement}. 
\begin{table}[htb]
\renewcommand{\arraystretch}{1.3}
\begin{tabular}{|p{2cm}|p{2cm}|p{4.7cm}|p{3cm}|p{2.7cm}|}
\hline
Stability & New \newline singularities & Newly \newline disallowed \newline configurations & 
Range of \newline$\alpha$-stability\newline cf. Remark \ref{R:19/29} & GIT \newline stability \newline range\\
\hline
\hline
DM-stable & $A_1$ && $\alpha\in[1,9/11)$ & \\
\hline
$A_2$-stable & $A_2$ & $A_1$-attached elliptic tails & $\alpha\in (9/11,7/10) $& \\
\hline
$A_3$-stable & $A_3$& $A_3$-attached elliptic tails, \newline $A_1$-attached elliptic bridges& $\alpha\in(7/10,2/3)$ & $\beta\in(1/7,4/13)$\\
\hline
$A_4$-stable & $A_4$ & & $\alpha=2/3$ & $\beta=4/13$\\
\hline
$A_4^{non-W}$-stable & & $p$ is a Weierstrass point & $\alpha\in(2/3,19/29)$ & $\beta\in(4/13,1/2)$\\
\hline
$A_5$-stable & $A_5$ & & $\alpha=19/29$  & $\beta=1/2$\\
\hline
\end{tabular}
\medskip
\caption{Stability notions for pointed genus two curves.}
\label{table-replacement}
\end{table}
\begin{remark}\label{R:19/29} For a general $(g,n)$, the notion of $\alpha$-stability 
for curves arising in the Hassett-Keel program for $\Mg{g,n}$ has been defined only for $\alpha>2/3-\epsilon$ 
(see e.g., \cite{afsflip-1}). 
There does exist a conjectural notion of $\frac{19}{29}$-stability (see \cite[Table 3]{AFS-modularity}),
 and as this paper illustrates, our $A_5$-stability coincides with it for $(g,n)=(2,1)$.
\end{remark}

\subsection{Relation to prior works} 
\label{S:prior}
As discussed in the introduction, moduli stacks $\Mg{2,1}(\alpha)$ for $\alpha>2/3-\epsilon$ have
been studied before in other contexts.  Here we give a brief overview of these works and relate
them to our construction.

Hyeon and Lee gave a GIT construction of $\Mg{2,1}(7/10)$ and $\Mg{2,1}(7/10-\epsilon)$ by 
taking GIT quotients of the Chow and Hilbert schemes of bi-log-canonical genus $2$ curves in $\PP^4$
\cite[Theorems 2.2 and 4.2]{hyeon-lee_genus4} in the course of their study of the Hassett-Keel program for $\Mg{4}$. Initially, we hoped to emulate their approach to the construction of $\Mg{2,1}(\alpha)$ for $\alpha\leq 2/3$. However, as we have discovered, the Chow point 
(in the sense of \cite[Section 4]{hyeon-lee_genus4}) of the ramphoid cuspidal atom $(C_{R},\infty)$ is unstable with respect to any linearization
and so Chow or asymptotic Hilbert stability cannot be used to construct the moduli stack $\Mg{2,1}(2/3)$ of $A_4$-stable curves.

In \cite{poli-W}, Polishchuk introduces two moduli stacks, $\Mg{2,1}(\cZ)$ and $\overline{\cU}^{ns}_{2,1}(2)$,
and proves that there is a morphism $\Mg{2,1}(\cZ) \to \overline{\cU}^{ns}_{2,1}(2)$ contracting
the Weierstrass divisor to a single point \cite[Theorem 2.4.5]{poli-W}.
Both of these stacks are smooth, Deligne-Mumford, and contain $\mathcal{M}_{2,1}$ as an open substack.
The first, $\Mg{2,1}(\cZ)$, is a stable modular compactification of $\mathcal{M}_{2,1}$ 
for an extremal assignment $\cZ$ over $\Mg{2,1}$ given by all unmarked irreducible components 
(cf. \cite[Definition 1.5 and Example 1.12]{smyth-modular}).
The second, $\overline{\cU}^{ns}_{2,1}(2)$, parameterizes non-special 
marked curves of genus $2$ in the sense of \cite{poli-grobner}. Here we identify these 
moduli stacks with those appearing in the Hassett-Keel program for $\Mg{2,1}$, and hence
with our GIT quotient stacks.

\begin{prop} We have isomorphisms $\Mg{2,1}(7/10-\epsilon)\simeq \Mg{2,1}(\cZ)$, and 
$\overline{\cU}^{ns}_{2,1}(2)\simeq \Mg{2,1}(2/3-\epsilon)$. 
\end{prop}
\begin{proof} We first prove that $\Mg{2,1}(7/10-\epsilon)\simeq \Mg{2,1}(\cZ)$.
The isomorphism types of $A_3$-stable curves that are not $\cZ$-stable are as follows:
\begin{enumerate}
\item[(a1)]  A union of two smooth rational
curves along $A_1$ and $A_3$-singularities, with one of the components marked.
This curve has no automorphisms (as the crimping parameter of the tacnode is killed by the automorphisms
of the unmarked component).  
\item[(a2)] A union of two smooth rational curves along three $A_1$-singularities, with one of the components marked. Such curves are Deligne-Mumford stable and also have no automorphisms.
\end{enumerate}
The isomorphism types of $\cZ$-stable curves that are not $A_3$-stable are as follows:
\begin{enumerate}
\item[(b1)]  An irreducible curve with a single two-branched singularity where one branch is a cusp,
and the other branch is smooth and transversal to the tangent space of the cusp (cf. the first
curve in \cite[\S2.3 Case IIb]{poli-W}). There is exactly one isomorphism class of such a curve,
and this curve has no automorphisms.
\item[(b2)] An irreducible curve with a single rational $3$-fold singularity (the coordinate cross singularity
of \cite[\S2.3 Case IIa]{poli-W}). Such curves also have no automorphisms. 
\end{enumerate}
We see that the complement of the union of (a1)-(a2) loci in $\Mg{2,1}(7/10-\epsilon)$ is isomorphic
to the complement of the union of (b1)-(b2) loci in $\Mg{2,1}(\cZ)$. Since these loci lie in the schematic
part of the respective Deligne-Mumford stacks, it suffices to show that the rational map 
$f\colon \Mg{2,1}(\cZ)\dashrightarrow \Mg{2,1}(7/10-\epsilon)$ has no indeterminacy. Towards this, we 
note that the Deligne-Mumford stable replacement of a curve $C$ of type (b2)
is a unique curve of type (a2), given by normalizing $C$ and attaching to $C$ a smooth rational 
component along the preimages of the rational $3$-fold singularity.  Hence the only possible indeterminacy point of 
$f$ is the (a1)-curve and the only possible indeterminacy point of $f^{-1}$ is the (b1)-curve. Since both stacks
are normal and schematic at these points, we conclude that $f$ is an isomorphism.  

The proof of $\overline{\cU}^{ns}_{2,1}(2)\simeq \Mg{2,1}(2/3-\epsilon)$ is exactly the same:
Away from the union of (a1)-(a2) loci $\Mg{2,1}(2/3-\epsilon)$ is isomorphic 
to the complement of the union of (b1)-(b2) loci in $\overline{\cU}^{ns}_{2,1}(2)$.
\end{proof}

\begin{remark} One can regard $\Mg{2,1}(7/10-\epsilon)$ (resp., $\Mg{2,1}(2/3-\epsilon)$) and 
$\Mg{2,1}(\cZ)$ (resp., $\overline{\cU}^{ns}_{2,1}(2)$) 
as two different approaches of compactifying the moduli 
space of crimping (i.e., attaching data) of a tacnodal singularity 
(cf. the discussion in \cite[p.468]{smyth-modular}). The presence of an $A_5$-singularity also shows
that $\Mg{2,1}(\alpha)$ are not stable modular compactifications of $\cM_{2,1}$ for $\alpha\leq 2/3$
(see \cite[Corollary 1.15]{smyth-modular}).
\end{remark}

We also note that, thanks to Corollary \ref{C:omega}, the stack 
$\Mg{2,1}(2/3-\epsilon)$ is isomorphic to the stack
$\mathcal{H}_5[4]$ of quasi-admissible genus $2$ hyperelliptic covers with at worst $A_4$-singularities,
as constructed in \cite[Proposition 4.2(1)]{fedorchuk-AD}, which leads to a different proof of 
\cite[Proposition 2.1.1]{poli-W} giving an isomorphism
\[
\Mg{2,1}(2/3-\epsilon) \simeq \overline{\cU}^{ns}_{2,1}(2) \simeq \mathcal{P}(2,3,4,5,6).
\]
\section{Proof of Main Theorem}
\label{S:proof}

Our GIT analysis of H-points of bi-log-canonical curves
is predicated on the study of GIT stability of (quadric sections of) degenerations 
of rational normal scrolls in $\PP^4$, and so in \S\ref{S:nets} we describe possible 
nets of quadrics in $\PP^4$ cutting out such degenerations. We recall the Hilbert-Mumford
numerical criterion as it applies to the elements of $\Grass(4,\Sym^2 V)\times \PP V^{\vee}$ 
in \S\ref{S:HM}, and use it in \S\ref{S:preliminary}
to narrow down potential $\beta$-stable pairs to those of H-pairs of bi-log-canonical curves.
In \S\ref{S:smooth}, we establish generic stability of smooth bi-log-canonical curves.
All these preliminary results are bootstrapped to give a proof of Theorem \ref{T:maintheorem} in
the latter parts of this section.

 \subsection{Orbits of nets of quadrics}
 \label{S:nets}
 It is easy to see that the $\SL(5)$-orbit of $N_1$ (given by \eqref{E:N1}) contains $N_2$ (given by \eqref{E:N2}) 
 in its closure. Indeed,
 \[
 \lim_{t\to 0}  
 \left(\begin{matrix} z_0 & z_2 & z_3 \\ tz_1+(1-t)z_2 & z_3 & z_4\end{matrix}\right)
 = \left(\begin{matrix} z_0 & z_2 & z_3 \\ z_2 & z_3 & z_4\end{matrix}\right).
 \]

In fact, we have a complete characterization of the orbit closure of 
$N_1$ in the Grassmannian of nets of quadrics in $\PP^4$. 
\begin{prop}\label{P:sl-orbits}The closure of the $\SL(5)$-orbit of $[N_1]\in \Grass(3,\Sym^2 V)$ is 
 \[
	 \overline{\SL(5)\cdot [N_1]}= \SL(5)\cdot [N_1] \cup \SL(5)\cdot [N_2] \cup \SL(5)\cdot [N_3] \cup \SL(5)\cdot [N_4] \cup \SL(5)\cdot [N_5]
 \]
 where 
 \begin{equation}\label{E:nets}
 \begin{aligned}
  N_3&= (z_0z_3, z_0z_4, z_2z_4-z_3^2), \\  
  N_4&= (z_0^2, z_0z_3, z_0z_4), \\ 
  N_5&= (z_0z_4-z_1z_3, z_1z_2, z_2z_4). 
 \end{aligned}
  \end{equation} 
 Moreover,
\[
 \overline{\SL(5)\cdot [N_2]}= \SL(5)\cdot [N_2]\cup \SL(5)\cdot [N_3] \cup \SL(5)\cdot [N_4].
 \]
 \end{prop}
 \begin{proof} Since we work in $\PP^4$, the net $N_1$ is the generic net of determinantal quadrics, i.e., quadrics given by 
 the minors of a $2\times 3$ matrix
 with linear entries. It follows by \cite[Theorem 2]{ellingsrud1981variety} that the 
 closure of the orbit $\SL(5)\cdot [N_1]$ is a smooth subvariety of 
 $\Grass(3, \Sym^2 V)$ whose boundary points also parameterize determinantal nets.
  With this in mind, let $D(L_0,\dots,L_5)\in  \Grass(3,\Sym^2 V)$ be the net of quadrics spanned
	by the $2\times 2$ minors of the matrix
 \[
 M(L_0,\dots,L_5)=\left(\begin{matrix} L_0 & L_2 & L_4 \\ L_1 & L_3 & L_5\end{matrix}\right),
 \]
 where $\{L_i\}_{i=0}^{5} \in V$.

	If the entries of $M(L_0,\dots,L_5)$ span $V$, then one verifies that
	\[
	\SL(5)\cdot D(L_0,\dots,L_5)=\SL(5)\cdot D(z_0, z_1, z_2, \sum_{i=0}^4 a_i z_i, z_3, z_4),\ \text{ where $a_i \in \CC$}.
	\]
	Suppose not all the $a_i$ are $0$. Then 
	\[
	\SL(5)\cdot D(z_0, z_1, z_2, \sum_{i=0}^4 a_i z_i, z_3, z_4)=\SL(5)\cdot D(z_0, z_1, z_2, z_3, z_3, z_4)=[N_1].
	\]
	If all the $a_i$ are $0$, then 
	$\SL(5)\cdot D(z_0, z_1,z_2, 0, z_3, z_4)=[N_5]$.

	Suppose the dimension of the span of the entries of $M(L_0,\dots,L_5)$ 
	is at most four. 
	Then $D(L_0,\dots,L_5)$ is (a cone over) a net of determinantal quadrics
	in $\PP^3$. Generically, this net cuts out a twisted cubic in $\PP^3$, and so 
	corresponds to $N_2$.
	The remainder of the proposition follows from the analysis of the 
	$\SL(4)$-orbit closure of $[N_2]\in \Grass\bigl(3, \HH^0(\PP^3,\cO_{\PP^3}(2))\bigr)$ performed in \cite{ellingsrud1981variety}.

 \end{proof}

\subsection{Numerical criterion}
\label{S:HM}
Our convention for the Hilbert-Mumford indices are as follows. Given a one-parameter 
subgroup (1-PS) $\rho$ of $\SL(5)$
acting diagonally on a basis $z_0,\dots, z_4$ of $V=\HH^0(\PP^4, \cO(1))$ with weights $w_0, \dots, w_4$, 
the Hilbert-Mumford
index $\mu_\rho(I)$ of $I \in \Grass(k, \Sym^m V)$ will be the sum of $\rho$-weights of the $k$
\emph{leading} (i.e., greatest $\rho$-weight) degree $m$ monomials of $I$ with respect to $\rho$. This also applies 
to the $m^{th}$ Hilbert points $(I_X)_m$ of closed subschemes $X\hra \PP^4$. 

In the case when $X=p=(a_0,\dots,a_4) \in \PP^4$ is a point and $m=1$, this translates into
\[
\mu_{\rho}(p):=-\min \{w_i \mid a_i\neq 0\}=\max\{-w_i \mid a_i\neq 0\}.
\]
We say that the point $(I, p)\in \Grass(4, \Sym^2 V) \times \PP V^{\vee}$ is $\rho$-unstable under linearization
$\cL_{\beta}$ if 
\[
\mu_{\rho}\bigl(I, p\bigr)=\mu_\rho(I)+\beta\mu_{\rho}(p)<0.
\]
The Hilbert-Mumford numerical criterion can then be translated as saying
that $(I,p)\in \bbH(\beta)$ if and only if $\mu_\rho(I)+\beta\mu_{\rho}(p)>0$ for all non-trivial 1-PS $\rho$ of $\SL(5)$. 

Note that if $\rho \subset \Stab\bigl(I, p\bigr)$ is a subgroup of the stabilizer of $(I,p)$, and 
$\rho^{-1}$ is its inverse subgroup, then
\[
\mu_{\rho}\bigl(I, p\bigr)=-\mu_{\rho^{-1}}\bigl(I, p\bigr),
\]
and so the necessary condition for $(I,p)$ to be semistable is that $\mu_{\rho}\bigl(I, p\bigr)=\mu_{\rho^{-1}}\bigl(I, p\bigr)=0$.

\subsection{Preliminary stability results}
\label{S:preliminary}
Consider a point $(I,p) \in \bbH \subset \Grass(4, \Sym^2 V)\times \PP^4$. Since $(I,p)$ lies in the closure
of the locus of H-pairs of smooth pointed genus $2$ curves and a generic such curve lies on
a smooth cubic scroll $S_1$, we conclude that 
there is a dimension $3$ linear subspace $N \subset I$ such that $N\in \overline{\SL(5)\cdot [N_1]}$ as an 
element of $\Grass(3, \Sym^2 V)$.  The orbit closure $\overline{\SL(5)\cdot [N_1]}$ has been described in Proposition \ref{P:sl-orbits}.
With this in mind, we have the following instability result:
 \begin{prop}
 Let $(I, p)\in \bar \bH \subset \Grass(4, \Sym^2 V)\times \PP^4$. 
 We have 
 that 
\begin{enumerate}
\item $(I, p) \notin \bbH(\beta) \ \text{for $\beta<1$, if $N_3 \subset I$,}$
\item $(I, p) \notin \bbH(\beta)  \ \text{for $\beta<2$, if $N_4 \subset I$,}$
\item $(I, p) \notin \bbH(\beta) \ \text{for $\beta<2$, if $N_5\subset I$.}$
\end{enumerate}
 \end{prop}
 \begin{proof}
 
 (1) Suppose that in coordinates $\{z_i\}_{i=0}^{4}$, we have that 
 $N_3=(z_0z_3, z_0z_4, z_2z_4-z_3^2) \subset I$. 
 Consider the 1-PS $\rho$ acting on $\{z_i\}_{i=0}^{4}$ with weights 
 $(-1,1,1,0, -1)$. Then
 \begin{align*}
 \mu_\rho(I) &\leq \mu_{\rho}(N_3)+2=(-1+(-2)+0)+2=-1, \\
 \mu_{\rho}(p) &\leq 1.
 \end{align*}
 Hence the pair $(I,p)$ is unstable
 with respect to $\cO_{\Grass}(1)\otimes \cO_{\PP^4}(\beta)$ for $\beta<1$.
 
 (2) Suppose that in coordinates $\{z_i\}_{i=0}^{4}$, 
 we have that $N_4=(z_0^2, z_0z_3, z_0z_4)\subset I$. Consider
the 1-PS $\rho$ acting on these coordinates with weights $(-1, 1,0,0,0)$. 
Then
 \begin{align*}
 \mu_\rho(I) &\leq \mu_{\rho}(N_4)+2=-2, \\
 \mu_{\rho}(p) &\leq 1.
 \end{align*}
 Hence the pair is unstable
 with respect to $\cO_{\Grass}(1)\otimes \cO_{\PP^4}(\beta)$ for $\beta<2$. 
 
(3) Suppose that in coordinates $\{z_i\}_{i=0}^{4}$, we have that $N_5=(z_1z_2, z_0z_4-z_1z_3, z_2z_4)\subset I$. Note that the point $p$ lies on the flat limit of the directrix of $N_1$ in $N_5$, which is
given by the line $z_2=z_3=z_4=0$.
Consider
the 1-PS $\rho$ acting on the coordinates $\{z_i\}_{i=0}^{4}$ with weights $(3, -1, -4 , 3, -1)$. 
Let $Q\in I \setminus N_5$ so that $I=Q+N_5$.
Then
 \begin{align*}
 \mu_\rho(I) &\leq \mu_{\rho}(N_5)+\mu_{\rho}(Q)\leq (-5+2-5)+6=-2, \\
 \mu_{\rho}(p) &\leq 1.
 \end{align*}
 Hence the pair $(I,p)$ is unstable
 with respect to $\cO_{\Grass}(1)\otimes \cO_{\PP^4}(\beta)$ for $\beta<2$. 
 \end{proof}

It follows from the above proposition that:
\begin{corollary}
\label{C:semistable-curves} For $\beta\in (0,1)$,
the semistable locus $\bbH(\beta)$ 
contains only pairs $((I_2)_C, p)$, where $C$ is either  
a quadric section of the smooth cubic scroll $S_1$ or
a quadric section of the singular cubic scroll $S_2$. 
\end{corollary}

From now on, we identify points of $\bbH(\beta)$ with the pointed curves they represent.  
\begin{lemma}
\label{L:planar} 
For $\beta\in (0,1)$, all curves in $\bbH(\beta)$ are locally planar.
\end{lemma} 
\begin{proof} In view of Corollary \ref{C:semistable-curves}, it suffices to show that $\bbH(\beta)$ 
contains no quadric sections
of $S_2$ passing through the vertex of the cone. Recall that the equation of $S_2$ is 
given by 
\[
N_2=(z_0z_3 -z_2^2, z_0z_4-z_2z_3, z_2z_4-z_3^2).
\]
Suppose that $I=N_2+Q$, where $Q\in I\setminus N_2$, defines a quadric section
passing through the vertex of $S_2$. Then $Q\in (z_0,z_2,z_3, z_4)$. 
Consider the 1-PS $\rho$ acting with weights $(-1,4,-1,-1, -1)$. Then
$\mu_{\rho}(I)\leq -3$, and $\mu_{\rho}(p)\leq 1$.
This shows that $(I,p)$ is unstable for $\beta<3$. 
\end{proof}

\begin{lemma}
\label{L:smooth-point} If $((I_C)_2, p)\in \bbH(\beta)$ for some $\beta\in (1/7, 1)$, 
then $p$ is a smooth point of $C$.
\end{lemma}
\begin{proof}
For $\beta>1/7$, we are going to destabilize the pair $((I_C)_2,p)$ where $C$ is a quadric section of the smooth scroll
meeting the directrix of the scroll in a point $p$ such that $p$ is a singular point of $C$. 
From the above discussion, this is a generic case, and so instability of all other pairs where $p$ is a singular point of $C$ 
will follow.
Assuming that the smooth scroll $S_1$ is cut out by the equations \eqref{E:N1}, that 
$p=[1:0:0:0:0]$ and that the tangent space to $C$ at $p$ is $z_3=z_4=0$, we see that 
the quadrics cutting out $C$ are 
\begin{equation}\label{E:singular-p}
\bigl(z_0z_3 -z_1z_2, z_0z_4-z_1z_3, z_2z_4-z_3^2, z_1^2+g(z_1,z_2,z_3,z_4)\bigr),
\end{equation}
Consider now the 1-PS $\rho$ acting with weights $(7, 2, 2, -3, -8)$ on the chosen coordinates. 
Then the 2nd Hilbert point of $C$ has $\mu_{\rho}((I_C)_2)\leq 4+(-1)+(-6)+4=1$, and $\mu_{\rho}(p)=-7$.  
The claim follows.
\end{proof} 
\begin{remark}[GIT quotients for $\beta<1/7$] Although we will not need this for the purposes of this paper,
one can show that for $\beta<1/7$ the generic  curves given by Equation \eqref{E:singular-p} 
and marked by $p=[1:0:0:0:0]$
are in fact GIT stable. The corresponding $\beta$-stable pairs in the GIT quotient parameterize irreducible 
arithmetic genus $2$ curves 
with a non-separating inner node $p$ and \emph{a choice} 
of a Cartier divisor of degree $2$ supported at $p$. Because linear 
equivalence
classes of such divisors form 
a one-dimensional family, a simple dimension count shows that such $\beta$-stable pairs 
form a codimension one locus in the GIT quotient 
$\bbH(\beta)\gitq \SL(5)$ for $\beta<1/7$. It follows that the rational map $\M_{2,1} \dashrightarrow \bbH(\beta)\gitq \SL(5)$ is not a contraction for $\beta<1/7$. Instead, this map
is a blow up of the codimension two locus of elliptic bridges in $\M_{2,1}$. 
 We conjecture that there 
is a diagram 
\begin{equation}
\begin{gathered}
\xymatrix{
\bbH(1/7-\epsilon)\gitq \SL(5) \ar[d] \ar[rd] &  & \bbH(1/7+\epsilon)\gitq \SL(5) \ar[ld]_{\simeq} \ar[d]^{\simeq} \\
\M_{2,1}(A_2)=\M_{2,1}(9/11-\epsilon) & \bbH(1/7)\gitq \SL(5)  & \M_{2,1}(7/10+\epsilon),
}
\end{gathered}
\end{equation}
where the triangle connecting the GIT quotients is a VGIT wall-crossing at $\beta=1/7$
and the isomorphism $\bbH(1/7+\epsilon)\gitq \SL(5) \simeq \M_{2,1}(7/10+\epsilon)$ is given
by our Theorem \ref{T:maintheorem}(A).

\end{remark}

 \subsection{Stability of smooth non-Weierstrass curves}
 \label{S:smooth}
 In this section, we prove the following stability result:
 \begin{prop}\label{P:smooth-nonW} Let $(C,p)$ be a bi-log-canonically embedded \emph{smooth} curve, 
 where $p$ is not a Weierstrass point.
 Then the H-point $h(C,p)$ is $\beta$-stable for all $\beta \leq 1/3$.
 \end{prop}

Let $I_2:=(I_C)_2$ be the linear system of quadrics cutting out the bi-log-canonical embedding of $(C,p)$. 
We begin with a preliminary observation:
\begin{claim}\label{common-sing} No two quadrics in $I_2$ have a common singular point.
\end{claim}
\begin{proof} Suppose they do. 
Then $C$ lies on a cone over a complete intersection of two quadrics in $\PP^3$, which is a
degree $4$ curve (possibly not integral). 
Since $C$ is a smooth scheme-theoretic intersection of the quadrics in $I_2$, 
it cannot pass through the vertex of the cone. It follows that
the projection of $C$ to $\PP^3$ must be supported on a non-degenerate integral curve whose
degree divides $6$. Since this curve lies on two linearly independent quadrics,  
we conclude that the projection is necessarily a twisted cubic. This contradicts the assumption 
that $(C,p)$ is not a Weierstrass curve.
\end{proof}

 \begin{proof}[{Proof of Proposition \ref{P:smooth-nonW}}]
 We use the following elementary fact repeatedly:
 \begin{claim}\label{C:opt}
 Suppose $a\geq b\geq c\geq d\geq e$ are rational numbers with $a+b+c+d+e=0$. Then 
 \[
 \lambda_1 a+\lambda_2 b+\lambda_3 c+\lambda_4 d+\lambda_5 e \geq 0
 \]
 for all non-negative rational numbers $\{\lambda_i\}_{i=1}^5$ satisfying 
 \[
 \sum_{i=1}^k \lambda_i > \frac{k}{5} \sum_{i=1}^5 \lambda_i \quad \text{for all $k=1,\dots, 4$}.
 \]
 \end{claim}
 We apply the Hilbert-Mumford numerical criterion. Suppose the 1-PS $\rho$ of $\SL(5)$ 
 acts with weights $a\geq b\geq c\geq d\geq e$ on a basis $z_0,\dots, z_4$
 of $\HH^0\bigl(C, \omega_C^2(2p)\bigr)$. 
 By the assumption $\beta\leq 1/3$, we always have
 \[
 \beta \mu_{\rho}(p) \geq -\frac{a}{3}.
 \]
 
 Let now $q$ be the point $z_1=z_2=z_3=z_4=0$, 
 let $\ell$ be the line
 $z_2=z_3=z_4=0$, and let $\pi$ be the plane $z_3=z_4=0$.  
 Let $m_1, m_2, m_3, m_4$ be 
 the initial monomials of the quadrics in $I_2$ with respect to the lexicographic order,
 and $Q_1,Q_2,Q_3,Q_4$ be respectively the generators of $I_2$ with those initial monomials.

Suppose first that $q \notin C$. Then $m_1=z_0^2$.
By Claim \ref{common-sing}, $\dim (I_2 \cap (z_1,z_2,z_3,z_4)^2)\leq 1$, and so 
we must have $m_2\geq z_0z_3$, and $m_3\geq z_0z_4$.
Furthermore, since $C$ does not lie on a rank $2$ quadric, we must 
have either $m_4\geq z_2z_3$, or,  after an appropriate change of basis, $Q_4=z_2z_4+z_3^2$. 
In the former case,
\[
\mu_{\rho}(I_2)\geq a^2+(a+d)+(a+e)+(c+d)=4a+c+2d+e,
\]
and so 
\[
\mu_{\rho}(C,p)= \mu_{\rho}(I_2)+\beta \mu_{\rho}(p)\geq \frac{11}{3} a+c+2d+e >0,
\]
by Claim \ref{C:opt}.
In the latter case, using the fact that no three linearly independent quadrics in $I_2$ 
can contain the plane $\pi$, we see that we must have a $z_2^2$ term in either $Q_2$ or $Q_3$. 
Therefore, 
\[
\mu_{\rho}(I_2)\geq 
2a+\frac{1}{2}((a+d)+2c)+(a+e)+\frac{1}{2}(c+e+2d)=\frac{7}{2}a+\frac{3}{2}c+\frac {3}{2}d+\frac{3}{2}e,
\]
and so 
\[
\mu_{\rho}(C,p)\geq \frac{19}{6}a+\frac{3}{2}c+\frac{3}{2}d+\frac{3}{2}e>0,
\]
by Claim \ref{C:opt}.
This concludes the proof of stability when $q\notin C$.

Suppose now that $q \in C$. Then because the curve $C$ is smooth at $q$, 
we have at least $3$ initial monomials
divisible by $z_0$, namely $m_1 \geq z_0z_2$, $m_2\geq z_0z_3$, and $m_3 \geq z_0z_4$. As before, we must 
have either $m_4\geq z_2z_3$, or $Q_4=z_2z_4+z_3^2$. Summarizing, 
\[
\mu_{\rho}(I_2)\geq (a+d)+(a+c)+(a+e)+\frac{1}{3}(2(c+e)+2d)=3a+\frac{5}{3}(c+d+e), 
\]
and so
\[
\mu_{\rho}(C,p)\geq \frac{8}{3} a+\frac{5}{3}(c+d+e)=a-\frac{5}{3}b.
\]
If $a>\frac{5}{3}b$, we are done. 
Suppose $a\leq \frac{5}{3}b$, and in particular $b>0$. 
Since not all quadrics can contain the line $\ell$, at least one of the quadrics contains 
a term $\geq z_1^2$. It follows that 
\[
\mu_{\rho}(I_2)\geq 2b+(a+c)+(a+e)+\frac{1}{3}((c+e)+4d)=2a+2b+\frac{4}{3}(c+d+e),
\]
and so \[
\mu_{\rho}(C,p)\geq \frac{5}{3}a+2b+\frac{4}{3}(c+d+e)=\frac{1}{3}a+\frac{2}{3}b >0.
\]

\end{proof}

\begin{prop}\label{P:beta-A4-stability}
For $\beta\in (1/7, 1/2]$, 
every $\beta$-stable pair in $\bbH$ is the H-point $h(C,p)$
where $(C,p)$ is a pointed
genus $2$ curve  
with at worst $A_5$-singularities, and no
		nodally or tacnodally attached elliptic tails or nodally attached elliptic bridges.
		Moreover, for $\beta\in (1/7, 1/2)$, the curve $(C,p)$ has at worst $A_4$-singularities
		and hence is $A_4$-stable.
\end{prop}
\begin{proof}
By Corollary \ref{C:semistable-curves} and Lemma \ref{L:planar}, for $\beta\in (0,1)$
every $\beta$-stable pair has form $((I_C)_2,p)$ where $C$ is a quadric 
section of the smooth cubic scroll $S_1$ or a quadric section of the singular cubic scroll $S_2$, 
not passing through the vertex. Moreover, for $\beta>1/7$,  Lemma \ref{L:smooth-point} implies that
$C$ is tangent to the
directrix of $S_1$ (resp., to a ruling of $S_2$) at a smooth point $p$. 
The adjunction formula now gives that 
\[
\omega^2_C(2p)=\cO_{S_i}(1)\vert_{C},
\]
and so $(C,p)$ is a bi-log-canonically embedded locally planar genus $2$ pointed curve.

The ruling of the scroll $S_i$ cuts out a complete linear system 
$\vert \omega_C\vert$ that defines a finite flat degree $2$ 
morphism to $\PP^1$. It follows that the singularities of $C$ are of type $A$, 
and hence at worst $A_5$ by genus considerations.  It also follows
that $C$ is either irreducible or has two rational
components. 

Assume that $\beta\in (1/7, 1/2)$. To prove that $C$ is $A_4$-stable, 
it remains to show that $C$ has no $A_5$-singularities. Suppose to the contrary. 
Then by genus considerations, 
$C$ is a union of two rational components meeting in a single $A_5$ singularity.  Reducible 
$\beta$-stable curves cannot lie on $S_2$ because such curves cannot be tangent
to a ruling at a smooth point. Therefore, $C$ must lie on $S_1$. 
In this case, the plane quartic model of $(C,p)$ is a union of a cuspidal cubic 
and a flex line, given by the equation $z(zy^2-x^3)=0$. The equations
cutting out $C$ in $\PP^4$ are then given by
\begin{equation}\label{A_5-curve}
I_{A_5}:=(z_0z_3 -z_1z_2, z_2z_4-z_3^2, z_0z_4-z_1z_3, z_1^2-z_0z_2),
\end{equation}
while the marked point $p$ is given by $z_1=z_2=z_3=z_4=0$.

We note that $I_{A_5}$ is fixed by the 1-PS of $\SL(5)$ acting on 
$z_0,\dots, z_4$ with weights $(-2,-1,0,1,2)$. The generators of $I_{A_5}$ 
as listed in \eqref{A_5-curve}
have weights $-1, 2, 0, -2$,
respectively. Recalling that $p=[1:0:0:0:0]$, we see that under the linearization $\cL_{\beta}$ the Hilbert-Mumford index of 
$(I_{A_5}, p)$ with respect to this one-parameter subgroup is $-1+2\beta$.  
The pair $(C,p)$ is then destabilized by the 1-PS $(-2,-1,0,1,2)$ for $\beta<1/2$, 
and so we are done.

\end{proof}
We will call the curve described by Equation \eqref{A_5-curve} the \emph{$A_5$-curve}.  As the following lemma shows,
every $A_4^{non-W}$-stable curve isotrivially specializes to the $A_5$-curve. 
\begin{lemma}
\label{L:A5-curve}
 The $\SL(5)$-orbit closure in $\bbH$ of every $A_4^{non-W}$-stable curve contains the orbit of the $A_5$-curve.
Furthermore, for $\beta\in(1/2,1)$, we have $\bbH(\beta)=\varnothing$.
\end{lemma}
\begin{proof} It suffices to show that all quadric sections of the smooth cubic scroll $S_1$ that
are simply tangent to the directrix $z_2=z_3=z_4=0$ isotrivially specialize to the $A_5$-curve and 
are unstable for $\beta>1/2$. Assuming that
$p=[1:0:0:0:0]$, the equation 
of such quadric section will necessarily be of the form
\[
(z_0z_3 -z_1z_2, z_2z_4-z_3^2, z_0z_4-z_1z_3, z_1^2-z_0z_2+Q(z_1,\dots,z_5)),
\]
where $Q(z_1,\dots,z_5) \in (z_2,z_3,z_4)$. Per our conventions from \S\ref{S:HM},
 under the linearization $\cL_{\beta}$, the Hilbert-Mumford index of the H-point of $(C,p)$ 
with respect to the one-parameter subgroup 
$(2,1,0,-1,-2)$ is $1-2\beta$.  Hence this H-point is unstable once $\beta>1/2$. 
The flat limit of $(C,p)$ under the above one-parameter subgroup is the $A_5$-curve, marked by $[1:0:0:0:0]$.
\end{proof}

Next we prove that $\beta=4/13$ and $\beta=1/2$ are the only values in the interval $(1/7, 1/2]$ at which 
the semistable loci could change.  This will prove that $\bbH(\beta)$ is constant for $\beta\in (1/7, 4/13)$ and for
$\beta\in (4/13, 1/2)$.
\begin{lemma}\label{L:Gm-action} Suppose $(C,p)$ is a $\beta$-semistable curve for 
some $\beta \in (1/7, 1/2]$ with $\GG_m$-action.
Then either $\beta=4/13$ and $(C,p)=(C_{R}, \infty)$ is the ramphoid cuspidal atom (cf. \S\ref{S:atom}), or $\beta=1/2$ and $(C,p)$ is the $A_5$-curve
given by Equation \eqref{A_5-curve} and marked by $[1:0:0:0:0]$.
\end{lemma}

\begin{proof} By Proposition \ref{P:beta-A4-stability} and Corollary \ref{C:omega}, 
for $\beta\in (1/7, 1/2]$, every $\beta$-semistable
curve is honestly hyperelliptic.
If $\kappa_C\colon C\to \PP^1$ is the canonical morphism, then the branch locus of $\kappa_C$
is $\GG_m$-invariant. Since $p$ is a smooth point of $C$ and $\kappa_C(p)$ is also $\GG_m$-invariant, 
we conclude that 
the branch locus consists either of $\kappa_C(p)$ and another point of multiplicity $5$, or of a single point of multiplicity $6$, which is distinct from $p$.
In the former case, we obtain
$(C,p) \simeq (C_{R}, \infty)$ as given by \eqref{E:RamphoidCusp} 
and in the latter case, we obtain the $A_5$-curve.

Recall from \S\ref{S:atom} (whose notation we keep) that 
$I_{C_R}$ is fixed by the 1-PS acting on 
$z_0,\dots, z_4$ with weights $(13, 8, 3, -7, -17)$, and that the generators of 
$I_{C_R}$ have weights $6, -4, -14, 16$,
respectively. Since $\infty=[1:0:0:0:0]$, we conclude that under linearization 
$\cL_{\beta}$ the Hilbert-Mumford index of 
$(I_{C_R}, \infty)$ with respect to this 1-PS is 
\[
(6+(-4)+(-14)+16)-13\beta=4-13\beta.  
\]
This shows that $(C_R, \infty)$ can be semistable
only for $\beta=4/13$.

Similarly, we saw in the proof of Proposition \ref{A_5-curve} that the $A_5$-curve $(C,p)$
is fixed by a 1-PS of $\SL(5)$ with respect to which the Hilbert-Mumford index of $(C,p)$
is $2\beta-1$ under the linearization $\cL_{\beta}$.  This shows that $(C, p)$ can be semistable
only for $\beta=1/2$.

\end{proof}
\begin{remark} Note that we do not yet claim that $(C_R, \infty)$ is semistable at $\beta=4/13$.
This will be proved later in Lemma \ref{L:rampcusp-413}.
\end{remark}

\subsection{Proof of Theorem \ref{T:maintheorem} Part (C)}
\label{part-C}
Suppose $\beta\in (4/13, 1/2)$.  
 We begin by showing that Weierstrass curves are destabilized for $\beta>4/13$.
 \begin{lemma}\label{L:W-unstable} Every Weierstrass curve isotrivially degenerates to the ramphoid cuspidal atom $(C_R, \infty)$, and is 
unstable with respect to $\cO_{\Grass}(1)\otimes \cO_{\PP^4}(\beta)$ for $\beta>4/13$.
 \end{lemma}
 
 \begin{proof}
 Recall from \S\ref{S:weierstrass} that every Weierstrass curve $(C,p)$ is given by the equations 
 \begin{align*}
 I_C &=(z_0z_3 -z_2^2, z_0z_4-z_2z_3, z_2z_4-z_3^2, 
 z_1^2-z_0z_2-c_3z_2z_3-c_2z_2z_4-c_1z_3z_4-c_0z_4^2),
 \\
 p&=[1:0:0:0:0], \notag
 \end{align*}
 after an appropriate choice of a basis $z_0, \dots z_4$ of $\HH^0\bigl(C,\omega_C^2(2p)\bigr)$. 
 Let $\rho(t)$ be the one-parameter subgroup of $\SL(5)$ acting as 
 \[
 \rho(t) \cdot (z_0, \dots z_4)=(t^{13}z_0, t^{8} z_1, t^{3} z_2, t^{-7} z_3, t^{-17} z_4).
 \] 
  Then the flat limit of $\rho(t)\cdot I_C$ as $t\to \infty$ is precisely $I_{C_R}$. Since
 $[1:0:0:0:0]$ is fixed by $\rho$, the first claim follows. 
 
Next, the leading monomials of 
$(I_C)_2$
have $\rho$-weights $6$, $-4$, $-14$, $16$, so that $\mu_\rho((I_C)_2)=4$. On the other hand, $\mu_\rho(p)=-13$.
The instability claim follows.
 \end{proof}

By Proposition \ref{P:beta-A4-stability},
every $\beta$-stable pair $(C,p)$ is $A_4$-stable. Applying Lemma \ref{L:W-unstable}, 
we see that $(C,p)$ is in fact $A_4^{non-W}$-stable.

By Proposition \ref{P:smooth-nonW}, all smooth non-Weierstrass 
curves $(C,p)$ are $\beta$-stable for $\beta \leq 1/3$. Since the semistable locus is constant
for $\beta\in (4/13, 1/2)$, we conclude that  all \emph{smooth} non-Weierstrass pointed genus $2$ curves
are in $\bbH(4/13, 1/2)$. 

To complete the proof of Theorem \ref{T:maintheorem} (C),
it remains to show that \emph{singular} $A_4^{non-W}$-stable curves are also
$\beta$-stable for $\beta\in (4/13, 1/2)$.  For this, we make use of the GIT stack of semistable plane quartics. Recall
from \cite{GIT} that a plane quartic $D$ is semistable with respect to the 
standard $\PGL(3)$-action if and only if $D$ has no multiplicity three singularities,
and $D$ is not a union of a plane cubic and its flex line.  Moreover, $D$ is strictly semistable if and only if $D$ has a tacnode, 
or is a double conic. A strictly semistable quartic with a closed orbit is either 
a union of two conics along two tacnodes 
or a double conic.

\begin{lemma} Suppose $(C,p)$ is an $A_4^{non-W}$-stable curve.
Let $\psi \colon C\to \PP^2$ be the morphism given by $\omega_C(2p)$,
constructed in Lemma \ref{L:cuspidal-quartic}. Then:
\begin{enumerate}
\item $\psi(C)$ is a GIT semistable plane quartic with a cusp at $\psi(p)$.
\item $\psi(C)$ is a strictly semistable quartic if and only if one of the following holds:
\begin{enumerate} 
\item $\psi(C)$ is an irreducible plane quartic with a cusp and a tacnode as singularities.
\item $\psi(C$) is an irreducible plane quartic with a cusp and $A_4$-singularity.
\end{enumerate} 
\end{enumerate}
\end{lemma} 
\begin{proof} 
(1) follows from Lemma \ref{L:cuspidal-quartic} and 
the classification of GIT semistable quartics;  and (2) from 
the classification of strictly semistable plane quartics.
\end{proof}

\begin{lemma}\label{L:A_4-curve} There is a unique isomorphism class of a 
plane quartic in $\PP^2$ with a type $A_2$ (a cusp) and a type $A_4$ (a ramphoid cusp)
singularities. 
\end{lemma}
\begin{proof}
Let $C$ be a plane quartic with
singularities of type $A_2$ and $A_4$. Then
by genus considerations, $C$ is irreducible and its normalization 
is a rational curve. Assume that the cusp of $C$ is $[0:0:1]$, that the 
ramphoid cusp is at $[0:1:0]$, that the line $z=0$ is the tangent cone
of the ramphoid cusp, and that the line $y=0$ is the tangent cone 
of the cusp. Then the normalization morphism $\PP^1 \to \PP^2$ is necessarily given by 
 \begin{equation*}
 [s:t] \mapsto [s^2 t^2: st^3+t^4: s^4].
 \end{equation*} 
 The equation of the image quartic in $\PP^2$ is then
\begin{equation}\label{A_4-quartic}
(yz-x^2)^2-x^3z=0.
\end{equation}
\end{proof}
\begin{definition}[$A_4$-curve]
\label{D:A4}
 We call the pointed genus $2$ curve obtained by normalizing the plane 
quartic described by Lemma \ref{L:A_4-curve} at the cusp, and marking the preimage of the cusp,
\emph{the $A_4$-curve}.

\end{definition}

Define by $Z$ the locally 
closed subvariety of semistable plane quartics consisting of semistable 
quartics with a cusp, and by $\bar Z$ its closure. 
\begin{prop}\label{P:GIT-quartics} 
We have that $Z\gitq \PGL(3) \simeq \bar Z \gitq \PGL(3)$. 
\end{prop} 
\begin{proof}
What we need to show is that for every polystable quartic $[D]$ in $\bar Z$, there exists a unique orbit 
$\PGL(3)\cdot [E]$ in $Z$ 
such that $[D] \in \overline{\PGL(3)\cdot [E]}$. 

If $[D] \in \bar Z$ is a stable quartic, then $D$ must be a reduced quartic 
with a cusp, and no tacnodes. In this case, $\PGL(3)\cdot [D]$ is the requisite orbit in $Z$.

Suppose $[D]$ is polystable with a positive-dimensional stabilizer. Then the existence and uniqueness
of $[E]$ follows by inspecting the basins of attraction of $D$ as worked out 
in \cite{hyeon-lee_genus3}. Namely, if $D$ is a double conic $(yz-x^2)^2=0$, then $E$ will be the 
cuspidal plane model of the $A_4$-curve, i.e., the $A_4$-quartic $(yz-x^2)^2-x^3z=0$ from Lemma \ref{L:A_4-curve}. 
If $D$ is a union of two conics along
two tacnodes, i.e., the quartic given by an equation $(yz-x^2)(yz-\alpha x^2)=0$, then
$E$ will be $(yz-x^2)(yz-\alpha x^2)-x^3z=0$. 

The coordinate-free proof of the existence and uniqueness of $[E]$ is obtained by observing
that the orbit (closure) of the tacnodal quartic $D$ which is not the $A_4$-curve is determined 
by a \emph{crimping parameter}:  If $\widetilde{D}$ is the normalization of $D$ at the tacnode, then $\widetilde{D}$
is an arithmetic genus $1$ curve with a hyperelliptic involution exchanging the preimages $q_1, q_2$
of the tacnode
of $D$.  The crimping parameter of $D$ is then given by a scalar defining a vector space automorphism 
$T_{q_1} \simeq T_{q_2} \simeq T_{q_1}$ where the first isomorphism is the canonical isomorphism given 
by the identification of $T_{q_i}$ with the tangent cone of the tacnode of $D$, and the second isomorphism 
is given by the hyperelliptic involution on $\widetilde{D}$.
\end{proof}
From Proposition \ref{P:GIT-quartics}, we immediately obtain the following two corollaries:
\begin{corollary}\label{C:corollary-GIT} Suppose $(C,p)$ and $(C',p')$ are $A_4^{non-W}$-stable curves,
and let $\psi\colon C\to \PP^2$ (resp., $\psi'\colon C'\to \PP^2$) be the plane cuspidal model
of $(C,p)$ (resp., $(C',p')$). Then 
\[
\overline{\SL(3)\cdot \psi(C)}=\overline{\SL(3)\cdot \psi'
(C')}
\] 
if and only if $\SL(3)\cdot \psi(C)= \SL(3)\cdot \psi'(C')$ if and only if $(C,p)\simeq (C',p')$. 
\end{corollary}
 
 \begin{corollary} 
 Every $A_4^{non-W}$-stable curve has $\beta$-stable H-point for $\beta\in (4/13, 1/2)$. 
 \end{corollary}
 \begin{proof}
 We have that 
\[
\bbH(\beta)  
\subset
\left\{h(C,p) \mid \text{$(C,p)$ is $A_4^{non-W}$-stable}\right\} 
\subset Z.
\]
Furthermore, by Corollary \ref{C:corollary-GIT}, we have birational morphisms of GIT quotients
\[
\bbH(\beta)\gitq \SL(5) \to Z\gitq \SL(3) \simeq \bar Z\gitq \SL(3),
\]
where the last isomorphism is given by Proposition \ref{P:GIT-quartics}.
If some $h(C,p)$, where $(C,p)$ is an $A_4^{non-W}$-stable curve, 
was not a $\beta$-stable pair, then 
$\bbH(\beta)\gitq \SL(5) \to Z\gitq \SL(3)$ would not be a surjective morphism,
which is absurd since both the domain and the target are projective.
\end{proof}
 
\subsection{Proof of Theorem \ref{T:maintheorem} Parts (B) and (A)}
We have seen in Lemma \ref{L:Gm-action} that the only wall in $(1/7, 1/2)$
at which the semistable locus changes is $\beta=4/13$. We will next show that stability indeed
changes at $\beta=4/13$ and determine the variation of GIT picture around this wall.
Let $\cW\subset \bbH$ be the locus of Weierstrass $A_4$-stable curves. 
Let \[
\cZ^{-}:=\bbH(4/13)\setminus \bbH(4/13-\epsilon),\  \text{ and } 
 \ \cZ^{+}:=\bbH(4/13)\setminus \bbH(4/13+\epsilon).
 \] 
Since the action of $\Aut(C_{R}, \infty)^{\circ}\simeq \GG_m$ on $\Def^1(C_{R},\infty)$ has
no weight-spaces with weight $0$, we conclude that 
\[
\cZ^{-} \cap \cZ^{+} = \SL(5)\cdot (C_{R},\infty).
\]
We give the following exact description of the VGIT chambers:
\begin{prop}[VGIT chambers] 
\label{P:VGIT}
 $\cZ^{+}=\cW$, and 
 $\cZ^{-}$ consists of a single orbit of the $A_4$-curve (cf. Definition \ref{D:A4}).
\end{prop}

By construction, the $A_4$-curve is $A_4^{non-W}$-stable, as it has a well-defined
cuspidal plane quartic model given by Equation \eqref{A_4-quartic}. It follows that the $A_4$-curve 
is $\beta$-stable for $\beta\in (4/13, 1/7)$ by the already established Part (C) of Theorem \ref{T:maintheorem}.
We next observe that
the $A_4$-curve isotrivially specializes to $C_R$:
\begin{lemma}
\label{L:A_4-unstable}
The $A_4$-curve isotrivially specializes to $(C_R,\infty)$ and is unstable for $\beta<4/13$.
\end{lemma}
\begin{proof}
 Under the standard embedding of $S_1$ into $\PP^4$ (cf. Equation \eqref{scroll-embedding}, 
 the $A_4$-curve, given by Equation \eqref{A_4-quartic}), is the intersection of the scroll $S_1$ with the quadric
 $
 (z_1-z_2)^2-z_0z_2=0. 
 $
 The four quadrics cutting out the $A_4$-curve are thus:
 \begin{equation}\label{E:equation-ramphoid-cusp}
z_0z_3 -z_1z_2, z_0z_4-z_1z_3, z_2z_4-z_3^2, (z_1-z_2)^2-z_0z_2,
 \end{equation}
 and the marked point $p$ is $[1:0:0:0:0]$.

 Now consider the 1-PS $\rho$ acting with weights $(-13, -8, -3, 7, 17)$ on 
 $z_0, z_1-z_2, z_2, z_3, z_4$.
 Then the flat limit of the $\rho$-translate of the $A_4$-curve as $t\to 0$ is precisely $C_R$. 
Moreover, this one-parameter subgroup destabilizes the $A_4$-curve for $\beta<4/13$.
 \end{proof}

Next, we analyze GIT stability of the ramphoid cuspidal atom $(C_{R},\infty)$ (cf. \S\ref{S:atom}).
 
 \begin{lemma}\label{L:rampcusp-413} The pair $(C_R, \infty)$ is strictly semistable with a closed orbit with respect to the linearization $\cO_{\Grass}(1)\otimes \cO_{\PP^4}(4/13)$, and unstable with respect to 
 $\cO_{\Grass}(1)\otimes \cO_{\PP^4}(\beta)$ for $\beta\neq 4/13$.
 \end{lemma}
 \begin{proof} 
 We give two proofs of this key lemma.  For the first, note that the $A_4$-curve is semistable for $\beta>4/13$
 (Part (C) of Theorem \ref{T:maintheorem}),
 but is unstable for $\beta<4/13$ (Lemma \ref{L:A_4-unstable}). This implies that for $\beta=4/13$, there must exist a strictly semistable
 point with a closed orbit. By Lemma \ref{L:Gm-action}, the only possibility for such a curve is the ramphoid cuspidal
 atom $(C_R, \infty)$. 
 
 For a different proof, completely independent of our previous GIT analysis, note that
$\HH^0\bigl(C_R,\omega_{C_R}^2(2\infty)\bigr)$ is a multiplicity-free representation of $\Aut^0(C_R, \infty)^{\circ}=\GG_m$
 and this $\GG_m$ acts diagonally on the  basis of $\HH^0\bigl(C_R,\omega_{C_R}^2(2\infty)\bigr)$ given by \eqref{E:bi-log-basis}. 
 Thus it suffices
 to verify stability with respect to the torus acting diagonally on the same basis (see \cite{morrison-swinarski}). 
 This is a straightforward calculation, which we omit. 
 \end{proof}

We are now ready to finish our analysis of the VGIT chambers around $(C_R,\infty)$. 
\begin{proof}[Proof of Proposition \ref{P:VGIT}] 
We have established that the ramphoid cuspidal atom $(C_R,\infty)$ is strictly semistable for $\beta=4/13$.
Let $O$ denote the orbit of $(C_R,\infty)$ in $\bbH$.  Note that $\bbH$ is smooth along this orbit, 
since the singularities of $C_R$ are planar. 
Since $\Aut(C_R,\infty)^{\circ}=\GG_m$, and this $\GG_m$ acts with non-zero weights on the first-order
deformation space of $(C_{R}, \infty)$, 
it follows by general VGIT that $\cZ^{-}\cup \cZ^{+}$ consist of curves that isotrivially specialize to
$(C_R,\infty)$ (cf. \cite[Lemma 1.2 and Lemma 4.3]{thaddeus-VGIT}), and that $\GG_m$ acts with negative (resp., positive) weights on the normal bundle
$N_{O/\cZ^{-}}$ (resp., $N_{O/\cZ^{+}})$ at $(C_R,\infty)$ (see \cite[Proposition 4.6]{thaddeus-VGIT}).

We have seen in Lemma \ref{L:W-unstable} that $\cW\subset \cZ^{+}$, and 
the already established Theorem \ref{T:maintheorem} Part (C) implies that $\cZ^{+} \subset \cW$. We conclude that $\cZ^{+}=\cW$.

By Luna's \'etale slice theorem \cite{luna}, $N_{O/\bbH, x}$ is identified with the first-order deformation 
space of $x=(C_R, \infty)$, and under this identification, 
$N_{O/\cZ^{+},x}$ is identified with the four-dimensional space $W$,
the tangent space to the Weierstrass locus. It follows that $N_{O/\cZ^{-}, x}$ must have dimension 
at most one.  On the other hand, 
by Lemma \ref{L:A_4-unstable}, the $A_4$-curve lies in $\cZ^{-}$, and so we conclude that $\cZ^{-}$
consists precisely of the orbit of the $A_4$-curve. 
\end{proof}

By VGIT, we have $\bbH(4/13+\epsilon)\subset \bbH(4/13)$ and 
by Lemma \ref{L:rampcusp-413}, 
we have that every Weierstrass $A_4$-stable curve has $4/13$-semistable H-point.
Part (B) of Theorem \ref{T:maintheorem} now follows from Proposition \ref{P:beta-A4-stability} and Part (C).                                                                                                                                                       

Since the only polystable curve at $\beta=4/13$ is the ramphoid cuspidal atom
$(C_R, \infty)$, we have that 
\[
\bbH(4/13-\epsilon)=\bbH(4/13) \setminus \cZ^{-}.
\]
Proposition \ref{P:VGIT} now gives Part (A) of Theorem \ref{T:maintheorem}. 

\subsection{Proof of Theorem \ref{T:maintheorem}  Part (D)} Part (D)  follows immediately from Part (C) and Lemma \ref{L:A5-curve}.
\
\subsection{Proof of Theorem \ref{T:maintheorem}  Part (E)}
Let $\cU_{2,1}(A)$ be the stack of pointed genus $2$ curves $(C,p)$ with at worst $A$ singularities
and ample $\omega_C(p)$. 
Then we have the following relations in the Picard group of $\cU_{2,1}(A)$:
\begin{align*}
10\lambda &=\delta_{\irr}+2\delta_{1,1}, \\
\kappa&=12\lambda-\delta_{\irr}-\delta_{1,1}.
\end{align*}
This follows from the same relations on $\Mg{2,1}$ (see \cite[Proposition~1.9]{arbarello1998calculating} and 
\cite{mumford1983towards}) and the fact that the complement of $\Mg{2,1}$ in $\cU_{2,1}(A)$ 
is of codimension at least $2$.

Restricting to the open substack of curves with no elliptic tails, we obtain 
$\delta=10\lambda$ and $\kappa=2\lambda$. Since for $\beta\in (1/7, 1/2]$, all 
$\beta$-stable pairs in $\bbH(\beta)$ are H-points of $A_4$ or $A_5$-stable curves, and the non-nodal locus in $\bbH(\beta)$
has codimension at least two, a standard computation, as in \cite{mumford-stability}, shows that the polarization $\cL_{\beta}$
(cf. Equation \eqref{beta-linearization})
descends to the following ample line bundle on the GIT quotient $\bbH(\beta)\gitq \SL(5)$:

\begin{equation}\label{E:GIT-polarization}
L_{\beta}:=\frac{1}{5}(-\lambda+8\psi)-\frac{\beta}{5}(3\lambda+\psi).
\end{equation}
Clearly, $L_{4/13}$ is a positive rational multiple 
of 
\[
(13\lambda-2\delta+\psi)+\frac{2}{3}\delta+\frac{1}{3}\psi.
\]
Recalling that also $K_{\overline{\cM}_{2,1}}=13\lambda-2\delta+\psi$ (see \cite[Theorem 2.6]{logan-kodaira}),
a standard discrepancy computation now shows that the isomorphism in Equation \eqref{lc-2/3} holds. 
The fact that \[
\bbH(4/13-\epsilon)\gitq \SL(5) \to  \bbH(4/13)\gitq \SL(5)
\]
is the contraction of the Weierstrass divisor follows from Proposition \ref{P:VGIT}.

The remaining identifications of $\bbH(\beta)\gitq \SL(5)$ with the log canonical models of $\Mg{2,1}$ follow by a similar computation. For example, since the Weierstrass divisor on $\Mg{2,1}$ has
class $3\psi-\lambda-\delta$ (see \cite[Proposition 6.70]{HM}) 
and it is contracted in $\bbH(4/13, 1/2)\gitq \SL(5)$, 
we have a further relation $\lambda=3\psi$ in the Picard group of $\bbH(4/13, 1/2)\gitq \SL(5)$.
It then follows that for $\beta=1/2$, we have
\begin{equation*}
L_{1/2}=\frac{3}{2}\psi-\frac{1}{2}\lambda=0.
\end{equation*}
At the same time, $L_{1/2}$ is a positive rational multiple 
of 
\[
(13\lambda-2\delta+\psi)+\frac{19}{29}\delta+\frac{10}{29}\psi,
\]
and so we obtain the isomorphism
\[
\bbH(1/2)\gitq \SL(5) \simeq \M_{2,1}(19/29) \simeq \text{\{point\}}.
\]

\bibliographystyle{amsalpha}
\bibliography{meta-bib,biblio}
\end{document}